\documentclass[11pt]{article}
\usepackage{amsmath}
\usepackage{amssymb}
\usepackage{amsthm}
\usepackage{mathabx}
\usepackage{caption}
\usepackage[usenames]{color}
\usepackage{amscd}
\usepackage{dsfont}
\usepackage{indentfirst}

\usepackage[colorlinks=true,linkcolor=blue,filecolor=red,
citecolor=webgreen]{hyperref}
\definecolor{webgreen}{rgb}{0,.5,0}

\def\C{{\mathds{C}}}

\def\R{{\mathbb{R}}}
\def\N{{\mathds{N}}}
\def\Z{{\mathds{Z}}}
\def\1{{\bf 1}}
\def\nr{{\trianglelefteq}}

\def\id{\operatorname{id}}
\def\lcm{\operatorname{lcm}}

\newtheorem{theorem}{Theorem}

\newtheorem{corollary}{Corollary}
\newtheorem{proposition}{Proposition}

\begin{document}
	
	\title{{\bf Counting subgroups of the groups $\Z_{n_1} \times \cdots \times \Z_{n_k}$: a survey}}
	\author{L\'aszl\'o T\'oth  \\
		Department of Mathematics \\
		University of P\'ecs \\
		Ifj\'us\'ag \'utja 6, 7624 P\'ecs, Hungary \\
		E-mail: {\tt ltoth@gamma.ttk.pte.hu}}
	\date{}
	\maketitle
	
	\centerline{New Frontiers in Number Theory and Applications, Trends in Mathematics,} 
	\centerline{Birkh\"auser, Cham, 2024, Guàrdia, J., Minculete, N., Savin, D., Vela, M.,} 
	\centerline{Zekhnini, A. (eds), pp. 385--409} 
	\centerline{\url{https://doi.org/10.1007/978-3-031-51959-8_18}}
	
	\begin{abstract}
		We present a survey of exact and asymptotic formulas on the number 
		of cyclic subgroups and total number of subgroups of the groups $\Z_{n_1} \times \cdots \times \Z_{n_k}$, 
		where $k\ge 2$ and $n_1,\ldots,n_k$ are arbitrary positive integers.
	\end{abstract}
	
	
	
	\tableofcontents
	
	\section{Introduction}
	
	Throughout this chapter we use the following notation: $\N:=\{1,2,\ldots\}$, $\N_0:=\N \cup \{0\}$; 
	the prime power factorization of $n\in \N$ is $n=\prod_p p^{\nu_p(n)}$, where the product is over the primes $p$ and all but a finite number of the exponents $\nu_p(n)$ are zero; 
	$\Z_n$ denotes the additive group of residue classes modulo $n$ (the cyclic group of order $n\in
	\N$); $(n_1,\ldots,n_k)$ and $\gcd(n_1,\ldots,n_k)$ denote the greatest common divisor of 
	$n_1,\ldots,n_k\in \N$, $[n_1,\ldots,n_k]$ and $\lcm(n_1,\ldots,n_k)$ denote the least common
	multiple of $n_1,\ldots,n_k\in \N$; $\id_t(n)=n^t$, $\sigma_t(n)=\sum_{d\mid n} d^t$ ($n\in \N, t\in \C)$; 
	$\sigma_0(n)=\tau(n)$ and $\sigma_1(n)=\sigma(n)$ denote the number and the sum of (positive) divisors of 
	$n\in \N$; $\varphi$ is Euler's arithmetic function; $\mu$ is the M\"obius 
	function; $f*g$ is the convolution of the functions $f,g:\N \to \C$; $\zeta$ is the Riemann zeta function;
	the sums $\sum_p$ and products $\prod_p$ and $\bigtimes_p$ are taken over the primes $p$.
	
	Let $G$ be a finite abelian group of order $n$. Let $s(G)$ and $c(G)$ denote the total number of subgroups\index{number of subgroups}
	of $G$ and the number of its cyclic subgroups\index{number of cyclic subgroups}, respectively. Also, let
	$G= \bigtimes_p G_p$ be the primary decomposition of $G$, where
	$|G_p| = p^{\nu_p(n)}$ ($p$ prime). Then
	\begin{equation} \label{prod_N}
		s(G)=\prod_p s(G_p), \qquad c(G)=\prod_p c(G_p).
	\end{equation}
	
	Therefore, the problem of counting the subgroups of $G$ reduces to $p$-groups. Formulas for the total number of subgroups
	of a given type of an abelian $p$-group were established by Bhowmik \cite{Bho1996}, Delsarte \cite{Del1948}, Miller \cite{Mil1904,Mil1939}, Shokuev 
	\cite{Sho1973}, Yeh \cite{Yeh1948}, and others. Also see 
	the papers by Stehling \cite{Ste1992} and Suzuki \cite{Suz1951}, and the monographs by Butler \cite{But1994} and Schmidt \cite{Sch1994}. One of these 
	formulas is in terms of the gaussian coefficients $\left[r\atop k\right]_p= \prod_{i=1}^k \frac{p^{r-k+i}-1}{p^i-1}$. Namely, let $G_{(p)}$ be a $p$-group 
	of rank $r$ and type $\lambda= (\lambda_1,\ldots,\lambda_r)$, that is,
	\begin{equation} \label{p_group_type}
		G_{(p)}\simeq \Z_{p^{\lambda_1}}\times \cdots \times \Z_{p^{\lambda_r}},
	\end{equation}
	with $\lambda_1\ge \ldots \ge \lambda_r\ge 1$, where $\lambda$ is a partition of $|\lambda|=\lambda_1+\cdots +\lambda_r$. 
	Then the number $s_{\mu}(G_{(p)})$ of
	subgroups of type $\mu$ (with $\mu \preceq \lambda$, meaning that $\mu_i\le \lambda_i$ for $1\le i\le r$) of $G_{(p)}$ is
	\begin{equation*} 
		s_{\mu}(G_{(p)})= \prod_{j=1}^{\lambda_1}
		p^{\mu'_{j+1}(\lambda'_j-\mu'_j)} \left[{\lambda'_j-\mu'_{j+1} \atop
			\mu'_j-\mu'_{j+1}} \right]_p,
	\end{equation*}
	where $\lambda'$ and $\mu'$ are the conjugates (according to the
	Ferrers diagrams) of $\lambda$ and $\mu$, respectively. Hence
	$s_{\mu}(G_{(p)})$ is a polynomial in $p$, with integer
	coefficients, depending only on $\lambda$ and $\mu$ (it is a sum of Hall polynomials).
	
	It follows that the total number of subgroups of $G_{(p)}$ is
	\begin{equation*} 
		s(G_{(p)}) = \sum_{0\le k\le |\lambda|} \sum_{\substack{\mu\preceq \lambda\\ |\mu|=k}}
		s_{\mu}(G_{(p)}),
	\end{equation*}
	and as a consequence, for abelian $p$-groups the total number of subgroups and the number of subgroups 
	of a given order are polynomials in $p$ with integer coefficients. In fact, it can be shown that 
	all coefficients of these polynomials are non-negative. However, it is a difficult task to make explicit these 
	polynomials, even in the case of $p$-groups of small rank. 
	
	There are more simple formulas for the number of cyclic subgroups of $p$-groups.
	Consider the $p$-group $G_{(p)}$ given by \eqref{p_group_type}. 
	The number of cyclic subgroups of order $p^{\nu}$
	of $G_{(p)}$ is
	\begin{equation*} 
		c_{p^{\nu}}(G_{(p)})= \frac{p^{j_{\nu}}-1}{p-1} p^{\lambda_{j_{\nu}
				+1}+\cdots+ \lambda_r+ \lambda_{r+1}+ (\nu-1)(j_{\nu}-1)},
	\end{equation*}
	where $\lambda_{j_{\nu}} \ge \nu > \lambda_{j_{\nu}+1}$ and
	$\lambda_{r+1}=0$. See Miller \cite{Mil1939}, Yeh \cite{Yeh1948}.
	
	If $G$ is a finite abelian group, then one can define the generalized counting function
	\begin{equation} \label{def_sigma_t}
		\sigma_t(G)=\sum_{H\le G} |H|^t,
	\end{equation}
	where the sum is over all subgroups $H$ of $G$ and $t\in \C$. If $t=0$, then $\sigma_0(G)=s(G)$,
	the number of all subgroups of $G$. Ramar\'e \cite{Ram2017} investigated the function $\sigma_t(G)$ 
	in the case when $G$ is a $p$-group of type \eqref{p_group_type}, proved an involved but fully 
	explicit formula for $\sigma_t(G)$, where $G$ is an arbitrary abelian $p$-group, and deduced 
	the rationality of a related generating series.
	
	Instead of $p$-groups we prefer to consider the group $G: =\Z_{n_1}\times \cdots \times \Z_{n_k}$, where $n_1,\ldots,n_k$ are arbitrary positive
	integers. Let $s(n_1,\ldots,n_k)$ and $c(n_1,\ldots,n_k)$ stand for the total number
	of subgroups of $G$ and the number of its cyclic subgroups, respectively. 
	According to \eqref{prod_N}, the functions $s(n_1,\ldots,n_k)$ and $c(n_1,\ldots,n_k)$ are multiplicative 
	functions of $k$ variables. Therefore, $s(n,\ldots,n)$ and $c(n,\ldots,n)$ are multiplicative in $n$, as 
	functions of a single variable. See the survey of the author \cite{Tot2014} on general properties of 
	multiplicative arithmetic functions of several variables. As well known, in the case $k=1$
	one has $s(n)=c(n)=\tau(n)$ for every $n\in \N$, $\tau(n)$ denoting the number of divisors of $n$.
	Also note that $\sigma_t(\Z_n)=\sum_{d\mid n} d^t =\sigma_t(n)$.
	
	Here we present a survey concerning the total number of subgroups (of a given order) and the number of cyclic subgroups (of a given order) of the group $\Z_{n_1}\times \cdots \times \Z_{n_k}$, where $k\ge 2$. Some new results are also included. We consider, in particular, 
	the cases $k=2$, $k=3$ and $k=4$, and present exact and asymptotic formulas on the functions 
	$s(n_1,\ldots,n_k)$ and $c(n_1,\ldots,n_k)$, where $n_1,\ldots,n_k\in \N$. We also discuss the representation of the subgroups in the cases $k=2$ and $k=3$, and offer some new results for $\sigma_t(G)$ in the case when $G=\Z_m \times \Z_n$ 
	($m,n\in \N$). We point out that there are known asymptotic formulas for the sums $\sum_{n_1,\ldots,n_k\le x} 
	c(n_1,\ldots,n_k)$ with $k\ge 2$, see Section \ref{Section_Number_cyclic} , and for $\sum_{n_1,n_2\le x} s(n_1,n_2)$, 
	see Section \ref{Section_k_2}. However, no asymptotic formula 
	for $\sum_{n_1,\ldots,n_k\le x} s(n_1,\ldots,n_k)$ is known the literature in the case $k\ge 3$.

	The new results are given together with their proofs. These are Propositions \ref{Th_cyclic_k}, \ref{lemma_quotient}, 
	Theorems \ref{Th_sigma_t}, \ref{Th_sigma_t_Dirichlet}, \ref{Theorem_subgroups}, \ref{Theorem_subgroups_order_k} and Corollaries \ref{Cor_subgroups}, 
	\ref{Cor_subgroups_order_k}. To be precise, Proposition \ref{lemma_quotient} and the results of Section \ref{Section_k_4} are included in the arXiv 
	preprint by the author \cite{Tot2016}, but have not been published elsewhere.

	\section{Number of cyclic subgroups} \label{Section_Number_cyclic}
	
	For $c(n_1,\ldots,n_k)$, the number of cyclic subgroups of the group $\Z_{n_1}\times \cdots \times \Z_{n_k}$ we have the 
	following compact formula, valid for every $k\ge 2$ and every $n_1,\ldots, n_k\in \N$: 
	\begin{equation} \label{total_number_cyclic_subgroups}
		c(n_1,\ldots,n_k)= \sum_{d_1\mid n_1,\ldots,d_k\mid n_k}
		\frac{\varphi(d_1) \cdots \varphi(d_k)}{\varphi(\lcm(d_1,\ldots,d_k))},
	\end{equation}
	where $\varphi$ is Euler's function. Formula \eqref{total_number_cyclic_subgroups} was derived by the author in \cite{Tot2011} using the orbit counting lemma (Burnside's lemma) and in
	\cite{Tot2012} by some simple number-theoretic arguments.
	
	Consequently, the number of all cyclic subgroups of an arbitrary finite abelian
	group $G$ can be described by formula \eqref{total_number_cyclic_subgroups}
	in a compact form. Namely, use the primary decomposition of $G$, mentioned above, or consider the invariant factor
	decomposition of $G$ given by
	\begin{equation} \label{invariant_decomp}
		G\simeq \Z_{n_1} \times \cdots \times \Z_{n_r}, 
	\end{equation}
	with $n_1,\ldots,n_r \in \N \setminus \{1\}$, $n_j \mid n_{j+1}$ ($1\le j\le r-1$).
	
	Note that if $|G|>1$, then the number $r$ in \eqref{invariant_decomp} is uniquely
	determined, it represents the minimal number of generators of the group $G$, and is called the rank of $G$.
	In the case of abelian $p$-groups this notion of rank recovers the rank mentioned above.
	
	As discussed in the Introduction, $c(n_1,\ldots,n_k)$ is a multiplicative function of $k$
	variables. This property is a direct consequence of convolutional formula \eqref{total_number_cyclic_subgroups}.
	
	Let $c_{\delta}(n_1,\ldots,n_k)$ denote the number of cyclic
	subgroups of order $\delta$ of the group $\Z_{n_1}\times \cdots \times \Z_{n_k}$. For every $n_1,\ldots,n_k\in \N$ and $\delta \mid \lcm(n_1,\ldots,n_k)$
	we have, see T\'oth \cite[Th.\ 2]{Tot2012}, 
	\begin{align*}
		c_{\delta}(n_1,\ldots,n_k) & = \frac1{\varphi(\delta)} \sum_{e\mid \delta}  \gcd(e,n_1) \cdots
		\gcd(e,n_k) \mu(\delta/e) \\
		& = \frac1{\varphi(\delta)} \sum_{\substack{d_1\mid n_1,\ldots, d_k\mid n_k
				\\ \lcm(d_1,\ldots,d_r)=\delta}} \varphi(d_1) \cdots \varphi(d_k).
	\end{align*}
	
	In the case $k=2$, identity \eqref{total_number_cyclic_subgroups}
	reduces to
	\begin{equation} \label{repr_c} 
		c(n_1,n_2)= \sum_{d_1\mid n_1,d_2\mid n_2} \varphi(\gcd(d_1,d_2)),
	\end{equation}
	which can also be written as 
	\begin{equation*} 
		c(n_1,n_2) = \sum_{d \mid \gcd(n_1,n_2)} (\mu*\varphi)(d) \tau(n_1/d)\tau(n_2/d).
	\end{equation*} 
	
	Note that for every prime $p$ and $1\le a\le b$ we have 
	\begin{equation*}
		c(p^a,p^b)= 2(1+p+p^2+\cdots +p^{a-1})+(b-a+1)p^a,
	\end{equation*} 
	and \eqref{repr_c} leads to the Dirichlet series
	representation
	\begin{equation*}
		\sum_{n_1,n_2=1}^{\infty} \frac{c(n_1,n_2)}{n_1^{s_1}n_2^{s_2}} = \frac{\zeta^2(s_1)\zeta^2(s_2)\zeta(s_1+s_2-1)}{\zeta^2(s_1+s_2)}.    
	\end{equation*}
	
	By using a variant of the hyperbola method, Nowak and T\'oth \cite{NowTot2013,NowTot2014} proved the asymptotic formula
	\begin{equation}  \label{asympt_c_2}
		\sum_{n_1,n_2\le x} c(n_1,n_2) =x^2 \sum_{j=0}^3 B_j (\log x)^j +
		O\left( x^{\frac{1117}{701}+\varepsilon}\right), 
	\end{equation}
	where $B_3= 12/\pi^4$, $B_2,B_1,B_0$ are explicit constants, and $1117/701\doteq 1.593437$. They also proved a similar formula with the same error term for the sum $\sum_{\substack{n_1,n_2\le x, (n_1,n_2)>1}} c(n_1,n_2)$ concerning the number of cyclic subgroups of the
	groups $\Z_{n_1}\times \Z_{n_2}$ having rank two, and pointed out that this error can be
	replaced by $O\left(x^{\frac{3-\theta}{2-\theta}+\varepsilon}\right)$, where $\theta$ is the exponent in Dirichlet's divisor
	problem. The $O$-term stated above was obtained from the sharpest result known at the time due to Huxley, namely
	from $\theta= 131/416$. It was pointed out by T\'oth and Zhai \cite{TotZha2018} that asymptotic formula \eqref{asympt_c_2} holds for the slight better exponent $4427/2779\doteq 1.593019$ by using the exponent $\theta=517/1648$ obtained by Bourgain and Watt \cite{BW}\footnote{This paper has been withdrawn by the second author in July 2023.}. Note that the limit of this approach is $11/7\doteq 1.571428$ with $\theta=1/4$.
	
	T\'oth and Zhai \cite{TotZha2018} proved, by using a two dimensional Perron formula and the complex integration method that
	the error term of \eqref{asympt_c_2} can be refined into $O\left(x^{3/2} (\log x)^{13/2}\right)$.
	Concerning the case $k=3$, the same authors \cite{TotZha2020} proved, again by using a multidimensional Perron formula and the complex integration method that
	\begin{equation*}  \label{asympt_c_3}
		\sum_{n_1,n_2,n_3\le x} c(n_1,n_2,n_3) =x^3 \sum_{j=0}^7 C_j (\log x)^j +
		O\left( x^{8/3+\varepsilon} \right), 
	\end{equation*}
	where $C_j$ ($0\le j \le 7$) are some explicit constants.
	
	As a consequence of certain more general results, it was obtained by Essouabri, Salinas Zavala 
	and T\'oth \cite[Cor.\ 5]{EST2022}, based on identity \eqref{total_number_cyclic_subgroups}, 
	by deep analytic methods that for every
	$k\ge 2$,
	\begin{equation} \label{asympt_c_k}
		\sum_{n_1,\ldots,n_k\le x} c(n_1,\ldots,n_k) = x^k Q(\log x) + O\left(x^{k-\beta}\right),
	\end{equation}
	where $Q(t)$ is a polynomial in $t$ of degree $2^k-1$ and $\beta$ is a positive real number.
	More exactly, by considering a large class of mutivariable multiplicative functions $f: \N^k\to \R_+$,
	it was proved in \cite{EST2022} the existence of the meromorphic continuation and some related 
	properties of the associated 
	multiple Dirichlet series 
	\begin{equation*}
		\sum_{n_1,\ldots,n_k=1}^{\infty} \frac{f(n_1,\ldots,n_k)}{n_1^{s_1}\cdots n_k^{s_k}},
	\end{equation*}
	and then formula \eqref{asympt_c_k} was deduced by combining those results and La Bret\`eche's multivariable 
	Tauberian theorems \cite[Ths.\ 1, 2]{bretechecompo}.
	
	In paper \cite{EST2022} a result similar to \eqref{asympt_c_k}, but with respect to the H\"older norm condition $(n_1^s+\cdots +n_k^s)^{1/s}\le x$, where $s\ge 1$ is a fixed real number, was also deduced. 
	Sui and Liu \cite{SL2020} obtained an asymptotic formula for the hyperbolic summation $\sum_{n_1n_2\le x} c(n_1,n_2)$, and studied the order of magnitude of the corresponding error term. Kiuchi and Saad Eddin \cite{KS2020} deduced an asymptotic formula for the weighted sum $\sum_{n_1n_2\le x} c(n_1,n_2)\log \frac{x}{n_1n_2}$.
	
	\section{Number of cyclic subgroups in the case $n_1=\cdots =n_k$}
	
	Let $f_k(n)$ denote the number of subgroups of index $n$ of the group $\Z^k$ (sublattices of the lattice $\Z^k$).
	It is known in the literature that for every $k,n \in \N$ one has 
	\begin{equation*}
		f_k(n) = \sum_{d_1d_2\cdots d_k=n} d_1^0d_2^1d_3^2\cdots d_k^{k-1},
	\end{equation*}
	\begin{equation*}
		\sum_{n=1}^{\infty} \frac{f_k(n)}{n^s} = \zeta(s)\zeta(s-1)\cdots \zeta(s-k+1).
	\end{equation*}
	
	These identities (having applications in crystallography) lead to the asymptotic formula
	\begin{equation} \label{asympt_f_k}
		\sum_{n\le x} f_k(n)= \frac{\zeta(2)\zeta(3)\cdots \zeta(k)}{k} x^k + O\left(x^{k-1} (\log x)^{2/3} \right),
	\end{equation}
	valid for every $k\ge 2$. See Grady \cite[Lemma 1]{Gra2005} (case $k=2$), Gruber \cite{Gru1997}, Knopfmacher \cite[Corollary\, (2.13)]{Kno1985}, Zou \cite{Zou2006}. 
	Note that the error term in \eqref{asympt_f_k} comes from the error due to Walfisz concerning the function $\sigma(n)=\sum_{d\mid n} d$, namely
	\begin{equation} \label{error_sigma}
		\sum_{n \le x} \sigma(n)= \frac{\pi^2}{12} x^2+O\left(x(\log x)^{2/3}\right).
	\end{equation}
	
	Now let $c_k(n)$ denote the number of cyclic subgroups of the group $\Z_n^k=\Z_n \times \cdots \times \Z_n $ ($k$ times). 
	Here $c_k(n)$ is multiplicative in $n$, and we have, see Haukkanen and Sivaramakrishnan \cite[Th.\ 4]{HauSiv1994}),
	\begin{equation} \label{c_k_n}
		c_k(n)= \sum_{d\mid n} \frac{\varphi_k(d)}{\varphi(d)},
	\end{equation}
	where $\varphi_k(n)= n^k \prod_{p\mid n} (1-1/p^k)$ is the Jordan function of order $k$ and
	$\varphi(n)=\varphi_1(n)$ is Euler's function. Note that \eqref{c_k_n} also follows
	from identity \eqref{total_number_cyclic_subgroups}.  Also see Schulte \cite{Sch1999}.
	
	In the case $k=2$ one has $c_2(n)=\sum_{d\mid n} \psi(d)$, where $\psi(n)= n \prod_{p\mid n} (1+1/p)$ is the Dedekind function.
	It was mentioned  in our arXiv preprint \cite{NowTot2013} that
	\begin{equation*}
		\sum_{n\le x} c_2(n)= \frac{5}{4}x^2 + O\left( x(\log x)^{5/3}\right),
	\end{equation*}
	the error term coming from the result of Walfisz for the function $\psi$ (the same error as in 
	\eqref{error_sigma}).
	
	Now we prove the following result, which is similar to \eqref{asympt_f_k}.
	
	\begin{proposition} \label{Th_cyclic_k} For every fixed $k\ge 3$,
		\begin{equation*}
			\sum_{n\le x} c_k(n)= \frac{\zeta(k)}{k}H(k)x^k + O\left(x^{k-1} \log x \right),
		\end{equation*}
		where
		\begin{equation*}
			H(k)= \prod_p \left(1+\frac1{p^2}+\frac1{p^3}+\cdots +\frac1{p^k}\right).
		\end{equation*}
	\end{proposition}
	
	\begin{proof} We use standard arguments. The Dirichlet series of $c_k(n)$ is, by \eqref{c_k_n}, 
		\begin{equation*}
			\sum_{n=1}^{\infty} \frac{c_k(n)}{n^s} = \zeta(s) \prod_p \sum_{\nu=0}^{\infty} \frac{\varphi_k(p^\nu)}{\varphi(p^\nu) p^{\nu s}}
		\end{equation*}
		\begin{equation*}
			= \zeta(s)\zeta(s-k+1) H(s), 
		\end{equation*}
		where
		\begin{equation*}
			H(s)= \sum_{n=1}^{\infty} \frac{h(n)}{n^s} = \prod_p \left( 1+ \frac1{p^{s-k+2}}+\frac1{p^{s-k+3}}+\cdots +\frac1{p^s} \right)
		\end{equation*}
		is absolutely convergent for $\Re s>k-1$. Hence
		\begin{equation*}
			c_k(n)=  \sum_{de=n} h(d)\sigma_{k-1}(e),
		\end{equation*}
		with $\sigma_{k-1}(n)=\sum_{d\mid n} d^{k-1}$. We have 
		\begin{equation*}
			\sum_{n\le x} c_k(n) = \sum_{d\le x} h(d) \sum_{e\le x/d} \sigma_{k-1}(e)=
			\sum_{d\le x} h(d) \left(\frac{\zeta(k)}{k}\left(\frac{x}{d}\right)^k+ O\left(\left(\frac{x}{d}\right)^{k-1}\right) \right)
		\end{equation*}
		\begin{equation*}
			= \frac{\zeta(k)}{k} x^k \sum_{d=1}^{\infty} \frac{h(d)}{d^k} + O\left(x^k \sum_{d>x} \frac{|h(d)|}{d^k} \right)+
			O\left(x^{k-1} \sum_{d\le x} \frac{|h(d)|}{d^{k-1}} \right).
		\end{equation*}
		
		The sum of the series in the main term is $H(k)$. Also, $h(p)=\frac{p^{k-1}-1}{p-1}\le 2p^{k-2}$ for any prime $p$ and $h(p^\nu)=0$ for any prime power 
		$p^\nu$ with $\nu \ge 2$. So, $h(n)\ll 2^{\omega(n)}n^{k-2}$, where $\omega(n)=\sum_{p\mid n} 1$. This gives
		\begin{equation*}
			\sum_{d>x} \frac{|h(d)|}{d^k} \ll \sum_{d>x} \frac{2^{\omega(d)}}{d^2}\ll \frac{\log x}{x}.
		\end{equation*}
		
		Furthermore,
		\begin{equation*}
			\sum_{d\le x} \frac{|h(d)|}{d^{k-1}} \le \prod_{p\le x} \sum_{\nu=0}^{\infty} \frac{|h(p^\nu)|}{p^{\nu(k-1)}}= 
			\prod_{p\le x} \left(1+ \frac{|h(p)|}{p^{k-1}}\right)
		\end{equation*}
		\begin{equation*}
			= \prod_{p\le x} \left(1+\frac1{p}+\frac1{p^2}+\cdots + \frac1{p^{k-1}}\right)\ll \log x 
		\end{equation*}
		by Mertens' theorem.
	\end{proof}

	\section{Number of all subgroups in the case $k=2$} \label{Section_k_2}
	
	Consider the group $\Z_m\times \Z_n$, where $m,n\in \N$. Note that $\Z_m\times \Z_n$ is isomorphic to $\Z_{(m,n)}\times \Z_{[m,n]}$, and it has rank two, assuming that $(m,n)>1$.
	
	Goursat's lemma for groups can be stated as follows:
	\vskip2mm
	
	{\sl i) Let $G$ and $H$ be arbitrary groups. Then there is a bijection between 
		the set of all subgroups of $G \times H$ 
		and the set of all $5$-tuples $(A, B, C, D, \Psi)$, where $B\, \nr \, A \le G$,
		$D\, \nr \, C\le H$ and $\Psi: A/B \to C/D$ is an isomorphism (here
		$\le$ denotes subgroup and $\nr$ denotes normal subgroup). More
		precisely, the subgroup corresponding to $(A, B, C, D, \Psi)$ is
		\begin{equation*}
			K= \{ (g,h)\in A\times C: \Psi(gB)=hD\}.
		\end{equation*}
		
		ii) Assume that $G$ and $H$ are finite groups and assume that the subgroup
		$K$ of $G \times H$ corresponds to the $5$-tuple $(A_K, B_K, C_K, D_K, \Psi_K)$ under 
		the above bijection. Then one has $|A_K|\cdot |D_K| =|K|=|B_K|\cdot |C_K|$.
	}
	\vskip2mm
	
	For the history of Goursat's lemma, its proof and discussion see the references given in T\'oth \cite{Tot2014Tatra}. Also see the paper by Bauer et. al. \cite{BSZ2015}.
	
	For every $m,n\in \N$ let
	
	\begin{equation} \label{def_J_m_n}
		J_{m,n}:=\left\{(a,b,c,d,\ell)\in \N^5: a\mid m, b\mid a, c\mid n, d\mid
		c, \frac{a}{b}=\frac{c}{d}, \right.
	\end{equation}
	\begin{equation*} \left.
		\ell \le \frac{a}{b}, \, \gcd\left(\ell,\frac{a}{b} \right)=1\right\},
	\end{equation*}
	and for $(a,b,c,d,\ell)\in J_{m,n}$ define
	\begin{equation} \label{def_K}
		K_{a,b,c,d,\ell}:= \left\{\left(i\frac{m}{a}, i\ell \frac{n}{c}+j\frac{n}{d}\right): 0\le i\le a-1, 0\le
		j\le d-1\right\}.
	\end{equation} 
	
	Using Goursat's lemma, the author \cite[Th.\ 3.1]{Tot2014Tatra} gave the following representation of the 
	subgroups of $\Z_m\times \Z_n$. 
	\vskip2mm
	
	{\sl
		Let $m,n\in \N$.
		
		i) The map $(a,b,c,d,\ell)\mapsto K_{a,b,c,d,\ell}$ is a bijection
		between the set $J_{m,n}$ and the set of subgroups of $(\Z_m \times
		\Z_n,+)$.
		
		ii) The invariant factor decomposition\index{invariant factor decomposition} of the subgroup
		$K_{a,b,c,d,\ell}$ is
		\begin{equation} \label{H_isom}
			K_{a,b,c,d,\ell} \simeq \Z_{\gcd(b,d)} \times \Z_{\lcm(a,c)},
		\end{equation}
		where $\gcd(b,d)\mid \lcm(a,c)$.
		
		iii) The order of the subgroup $K_{a,b,c,d,\ell}$ is $ad$ and its
		exponent is $\lcm(a,c)$.
		
		iv) The subgroup $K_{a,b,c,d,\ell}$ is cyclic if and only if
		$\gcd(b,d)=1$.
	}
	\vskip2mm
	
	The next figure represents the subgroup $K_{6,2,18,6,1}$ of $\Z_{12}\times \Z_{18}$. It has order $36$ and is isomorphic to $\Z_2\times \Z_{18}$.
	
	\begin{table}
		\centerline{$
			\begin{array}{ccccccccccccc}
				17 & \cdot & \cdot & \cdot & \cdot & \bullet & \cdot & \cdot & \cdot & \cdot & \cdot & \bullet & \cdot \\
				16 & \cdot & \cdot & \bullet & \cdot & \cdot & \cdot & \cdot & \cdot & \bullet & \cdot & \cdot & \cdot \\
				15 & \bullet & \cdot & \cdot & \cdot & \cdot & \cdot & \bullet & \cdot & \cdot & \cdot & \cdot & \cdot \\
				14 & \cdot & \cdot & \cdot & \cdot & \bullet & \cdot & \cdot & \cdot & \cdot & \cdot &
				\bullet & \cdot \\
				13 & \cdot & \cdot & \bullet & \cdot & \cdot & \cdot & \cdot & \cdot & \bullet & \cdot & \cdot & \cdot \\
				12 & \bullet & \cdot & \cdot & \cdot & \cdot & \cdot & \bullet & \cdot & \cdot & \cdot & \cdot & \cdot \\
				11 & \cdot & \cdot & \cdot & \cdot & \bullet & \cdot & \cdot & \cdot & \cdot & \cdot & \bullet & \cdot \\
				10 & \cdot & \cdot & \bullet & \cdot & \cdot & \cdot & \cdot & \cdot & \bullet & \cdot & \cdot & \cdot \\
				9 & \bullet & \cdot & \cdot & \cdot & \cdot & \cdot & \bullet & \cdot & \cdot & \cdot & \cdot & \cdot \\
				8 & \cdot & \cdot & \cdot & \cdot & \bullet & \cdot & \cdot & \cdot & \cdot & \cdot & \bullet & \cdot \\
				7 & \cdot & \cdot & \bullet & \cdot & \cdot & \cdot & \cdot & \cdot & \bullet & \cdot & \cdot & \cdot \\
				6 & \bullet & \cdot & \cdot & \cdot & \cdot & \cdot & \bullet & \cdot & \cdot & \cdot & \cdot & \cdot \\
				5 & \cdot & \cdot & \cdot & \cdot & \bullet & \cdot & \cdot & \cdot & \cdot & \cdot & \bullet & \cdot \\
				4 & \cdot & \cdot & \bullet & \cdot & \cdot & \cdot & \cdot & \cdot & \bullet & \cdot & \cdot & \cdot \\
				3 & \bullet & \cdot & \cdot & \cdot & \cdot & \cdot & \bullet & \cdot & \cdot & \cdot & \cdot & \cdot \\
				2 & \cdot & \cdot & \cdot & \cdot & \bullet & \cdot & \cdot & \cdot & \cdot & \cdot &
				\bullet & \cdot \\
				1 & \cdot & \cdot & \bullet & \cdot & \cdot & \cdot & \cdot & \cdot & \bullet & \cdot & \cdot & \cdot \\
				0 & \bullet & \cdot & \cdot & \cdot & \cdot & \cdot & \bullet & \cdot & \cdot & \cdot & \cdot & \cdot \\
				& 0\text{ } & 1\text{ } & 2\text{ } & 3\text{ } & 4\text{ } & 5\text{ } & 6
				\text{ } & 7\text{ } & 8\text{ } & 9\text{ } & 10 & 11
			\end{array}
			$}
	\end{table}
	
	According to the above properties, the total number $s(m,n)$ of
	subgroups of $\Z_m \times \Z_n$ can be obtained by counting the
	elements of the set $J_{m,n}$. We have the following simple compact formulas, see the author 
	\cite[Th.\ 4.1]{Tot2014Tatra}. For every $m,n\in \N$, $s(m,n)$ is given by
	\begin{align} \label{total_number_subgroups}
		s(m,n) & = \sum_{i\mid m, j\mid n} \gcd(i,j) \\
		\label{total_number_subgroups_var}
		& = \sum_{d \mid \gcd(m,n)} \varphi(d) \tau\left(\frac{m}{d} \right)
		\tau \left(\frac{n}{d} \right).
	\end{align}
	
	Representation and properties of the subgroups of $\Z_m \times \Z_n$, including
	identities \eqref{total_number_subgroups} and \eqref{total_number_subgroups_var}
	were also obtained by Hampejs, Holighaus, T\'oth and Wiesmeyr \cite[Th.\ 4]{HHTW2014} by tooking 
	a different approach.
	
	It is a direct consequence of \eqref{total_number_subgroups} or \eqref{total_number_subgroups_var}
	that for every prime $p$ and $1\le a\le b$, 
	\begin{equation} \label{s_prime_pow}
		s(p^a,p^b)= \frac{(b-a+1)p^{a+2}-(b-a-1)p^{a+1}-(a+b+3)p+(a+b+1)}{(p-1)^2}. 
	\end{equation}
	
	Formula \eqref{s_prime_pow} was also deduced, applying Goursat's
	lemma for groups by C\u{a}lug\u{a}reanu \cite[Sect.\ 4]{Cal2004}
	and Petrillo \cite[Prop.\ 2]{Pet2011}, using the concept of
	the fundamental group lattice by T\u{a}rn\u{a}uceanu \cite[Prop.\
	2.9]{Tar2007}, \cite[Th.\ 3.3]{Tar2010}, and from a general recurrence  
	relation by Ramar\'e \cite[Eq.\ (10)]{Ram2017}.
	
	We also remark that \eqref{total_number_subgroups} is a special case of an identity representing the total number of subgroups of a class of groups formed as
	cyclic extensions of cyclic groups, deduced by Calhoun \cite{Cal1987} and having a laborious proof. 
	
	Let $s_{\delta}(m,n)$ denote the number of subgroups of order $\delta$
	of $\Z_m \times \Z_n$. For every $m,n,\delta \in \N$ such that $\delta \mid mn$,
	\begin{align*}
		s_{\delta}(m,n) & = \sum_{\substack{i\mid \gcd(m,\delta)\\ j\mid \gcd(n,\delta)\\ \delta \mid ij}} \varphi(ij/\delta))
		\\
		& = \sum_{d\mid \gcd(m,n,\delta)} \varphi(d) N(m/d,n/d,\delta/d),
	\end{align*}
	where $N(a,b,c)$ is the number of solutions $(x,y,z,t)\in \N^4$ of the system of equations
	$xy=a, zt=b, xz=c$, see T\'oth \cite[Th.\ 4.3]{Tot2014Tatra}, Hampejs, Holighaus, T\'oth and Wiesmeyr
	\cite[Th.\ 4]{HHTW2014}.
	
	For the number of subgroups of $\Z_m \times \Z_n$ with a given
	isomorphism type $\Z_A \times \Z_B$ we have the following
	formula, see the author \cite[Th.\ 4.5]{Tot2014Tatra}. 
	\vskip2mm
	
	{\sl Let $m,n\in \N$ and let $A,B\in \N$ such that $A\mid B$, $AB\mid mn$. Let $A\mid \gcd(m,n)$.
		Then the number $N_{A,B}(m,n)$ of subgroups of $\Z_m \times \Z_n$, which are isomorphic to $\Z_A \times \Z_B$ is given by
		\begin{equation*} 
			N_{A,B}(m,n)= \sum_{\substack{i\mid m, j\mid n \\ AB \mid ij\\
					\lcm(i,j)= B}} \varphi \left(\frac{ij}{AB}\right),
		\end{equation*}
		where $\varphi$ is Euler's function. If $A\nmid \gcd(m,n)$, then $N_{A,B}(m,n)=0$.
	}
	\vskip2mm
	
	The following result is concerning the quotient group
	$(\Z_m\times \Z_n)/K_{a,b,c,d,\ell}$. 
	
	\begin{proposition} \label{lemma_quotient} 
		Let $m,n\in \N$. For every subgroup $K_{a,b,c,d,\ell}$
		defied by \eqref{def_K}, we have the invariant factor decomposition
		\begin{equation*} 
			(\Z_m\times \Z_n)/K_{a,b,c,d,\ell} \simeq \Z_{\gcd
				\left(\frac{m}{a},\frac{n}{c}\right)} \times \Z_{\lcm
				\left(\frac{m}{b},\frac{n}{d}\right)},
		\end{equation*}
		where $\gcd \left(\frac{m}{a},\frac{n}{c}\right) \mid \lcm
		\left(\frac{m}{b},\frac{n}{d}\right)$.
	\end{proposition} 
	
	\begin{proof} Let $K:= K_{a,b,c,d,\ell}$. The group $\Z_m\times \Z_n$ is
		generated by $(0,1)$ and $(1,0)$. Therefore, the quotient group
		$\Z_m\times \Z_n/K$ is generated by $(0,1)+K$ and $(1,0)+K$.
		
		First we show that the order of $(0,1)+K$ is $\frac{n}{d}$. Indeed, this
		follows from the following properties:
		
		$\bullet$ $t(0,1)\in K$ if and only if there is $(i,j)\in \{0,1,\ldots, a-1\}
		\times \{0,1,\ldots,d-1\}$ such that $(0,t)=(i\frac{m}{a}, i\ell
		\frac{n}{c} +j\frac{n}{d})$, where the last condition is equivalent, in turn, to $i=0$ and
		$t\equiv j\frac{n}{d}$ (mod $n$); 
		
		$\bullet$ the least such $t\in \N$ is $\frac{n}{d}$.
		
		Next we show that the order of $(1,0)+K$ is $\frac{m}{b}$. Indeed,
		observe that
		
		$\bullet$ $u(1,0)\in K$ if and only if there is $(i,j)\in \{0,1,\ldots, a-1\}
		\times \{0,1,\ldots,d-1\}$ such that $(u,0)=(i\frac{m}{a}, i\ell
		\frac{n}{c} +j\frac{n}{d})$;
		
		$\bullet$ therefore, $u\equiv i\frac{m}{a}$ (mod $m$) and $i\ell
		\frac{n}{c}+j\frac{n}{d} \equiv 0$ (mod $n$), that is
		$\frac{n}{c}(i\ell +j\frac{c}{d}) \equiv 0$ (mod $n$), $i\ell
		+j\frac{c}{d} \equiv 0$ (mod $c$), the latter is a linear congruence
		in $j$ and it turns out that for a fixed $i$ it has solution in $j$
		if $\gcd\left(\frac{c}{d},c\right)=\frac{c}{d} \mid i\ell$, that is
		$\frac{c}{d }\mid i$, since $\gcd\left(\ell,\frac{c}{d}\right)=1$
		by \eqref{def_J_m_n};
		
		$\bullet$ the least such $i\in \N$ is $i=\frac{c}{d}$;
		
		$\bullet$ finally, the least such $u\in \N$ is $u=\frac{c}{d}\cdot \frac{m}{a}=\frac{a}{b}\cdot
		\frac{m}{a}= \frac{m}{b}$, representing the order of $(1,0)+K$.
		
		We deduce that the exponent of $(\Z_m \times \Z_n)/K$ is
		$\lcm(\frac{m}{b}, \frac{n}{d})$.
		
		Now, $\Z_m \times \Z_n \simeq \Z_{\gcd(m,n)} \times \Z_{\lcm(m,n)}$
		is a group of rank $\le 2$. Therefore, the quotient group $(\Z_m
		\times \Z_n)/K$ has also rank $\le 2$. That is, $(\Z_m \times
		\Z_n)/K \simeq \Z_v \times \Z_w$ for unique $v$ and $w$, where
		$v\mid w$ and $vw=\frac{mn}{ad}$, since $|K|=ad$. Here the exponent of $\Z_v \times \Z_w$ is $w$, and we 
		obtain that the exponent of $(\Z_m \times \Z_n)/K$ is $\lcm(\frac{m}{b}, \frac{n}{d})=w$. Hence
		\begin{align*}
			v & =  \frac{mn}{adw}= \frac{mn}{ad\lcm (\frac{m}{b}, \frac{n}{d})}=
			\frac{mn}{ad \lcm (\frac{mn}{bn}, \frac{mn}{dm})} = \frac{mn}{ad
				\frac{mn}{\gcd(bn, dm)}} \\
			& =  \frac{\gcd(bn,dm)}{ad}= \gcd\left( \frac{n}{d}\cdot
			\frac{b}{a},\frac{m}{a}\right) = \gcd\left( \frac{n}{d}\cdot
			\frac{d}{c},\frac{m}{a}\right) = \gcd\left( \frac{m}{a},
			\frac{n}{c}\right).
		\end{align*}
		
		This completes the proof.
	\end{proof}  
	
	Based on formulas \eqref{total_number_subgroups} and \eqref{total_number_subgroups_var}, Nowak and T\'oth \cite{NowTot2013,NowTot2014} proved the following results. For every $z,w\in \C$ with $\Re z>1, \Re w>1$,
	\begin{equation*} 
		\sum_{m,n=1}^{\infty} \frac{s(m,n)}{m^z n^w}=
		\frac{\zeta^2(z)\zeta^2(w)\zeta(z+w-1)}{\zeta(z+w)},
	\end{equation*}
	and 
	\begin{equation} \label{ss}
		\sum_{m,n\le x} s(m,n) = x^2 \sum_{r=0}^3 A_r (\log x)^r + O\left(x^{\frac{1117}{701}+\varepsilon} \right),
	\end{equation} 
	for every fixed $\varepsilon>0$, where $A_3=2/\pi^2$, $A_2,A_1,A_0$ are explicit constants, and $1117/701\doteq 1.593437$. Nowak and T\'oth \cite{NowTot2013,NowTot2014} also proved a similar formula with the same error term for the sum $\sum_{\substack{n_1,n_2\le x, (n_1,n_2)>1}} s(n_1,n_2)$ concerning the number of all subgroups of the
	rank two groups $\Z_{n_1}\times \Z_{n_2}$. The error term of \eqref{ss} is the same as of \eqref{asympt_c_2}, and can be replaced by $O\left(x^{\frac{3-\theta}{2-\theta}+\varepsilon}\right)$, 
	where $\theta$ is the exponent in Dirichlet's divisor problem. T\'oth and Zhai \cite{TotZha2018} 
	proved, by using a two dimensional Perron formula and the complex integration method that
	the error term of \eqref{ss} can be refined into $O\left(x^{3/2} (\log x)^{13/2}\right)$.
	
	Sui and Liu \cite{SL2020} deduced an asymptotic formula for the hyperbolic summation $\sum_{mn\le x} s(m,n)$, and studied the order of magnitude of the corresponding error term.
	Kiuchi and Saad Eddin \cite{KS2020} deduced an asymptotic formula for the weighted sum $\sum_{mn\le x} s(m,n) \log \frac{x}{mn}$.
	
	Sui and Liu \cite{SL2022} proved by a method which is elementary in essence, but
	using some preliminary results of analytic nature, that
	\begin{equation*} 
		\sum_{\substack{m,n\in \N\\ m^2+n^2 \le x}} s(m,n) = x P(\log x) + O\left(x^{17/22+\varepsilon}\right),
	\end{equation*}
	where $P(t)$ is a polynomial in $t$ of degree three. 
	By using a two dimensional Perron formula and the complex integration method,
	the same authors \cite{SL2022k} proved that for every fixed $k\ge 2$,
	\begin{equation*} 
		\sum_{m,n\le x} s(m^k,n^k) = C_kx^{k+1} + O\left(x^{k+1/2+\varepsilon}\right),
	\end{equation*}
	where $C_k$ is an absolute constant. 
	
	It was mentioned in our arXiv preprint \cite{NowTot2013} that
	\begin{equation} \label{asympt_s_n_n}
		\sum_{n\le x} s(n,n)= \frac{5\pi^2}{24}x^2 + O\left( x(\log x)^{8/3}\right),
	\end{equation}
	the error term coming from the result of Walfisz for the Dedekind function 
	$\psi(n)=n\prod_{p\mid n} (1+1/p)$.
	
	T\u{a}rn\u{a}uceanu and T\'oth \cite[Prop. 3.2]{TarTot2017} proved that the sum of exponents of the subgroups of $\Z_m\times \Z_n$ is $\sigma(m)\sigma(n)$ for every $m,n\in \N$. Therefore, the arithmetic mean of exponents of the subgroups of $\Z_m\times \Z_n$ is
	$\sigma(m)\sigma(n)/s(m,n)$. Now consider the case $m=n$, and  let $AE(n)$ stand for the arithmetic mean of exponents of the
	subgroups of $\Z_n\times \Z_n$. We have
	\begin{equation*} \label{A}
		AE(n)= \frac{\sigma(n)^2}{s(n,n)},
	\end{equation*}
	and, see same authors \cite[Prop.\ 3.3]{TarTot2017},
	\begin{align*} 
		\sum_{n\le x} AE(n) = \frac{C}{2} x^2 + O\left(x \log^3  x \right),
	\end{align*}
	where
	\begin{align*}
		C:=  \prod_p \left(1-\frac1{p}\right) \sum_{\nu=0}^{\infty} \frac{(p^{\nu+1}-1)^2}{p^{2\nu}(p^{\nu+2}+p^{\nu+1}- (2\nu+3)p+2\nu+1)}.
	\end{align*}
	
	A generalization of the function $s(m,n)$ is $\sigma_t(m,n)$, defined as the sum of 
	$t$-th powers of the orders of subgroups of $\Z_m \times \Z_n$, that is,
	\begin{equation*}
		\sigma_t(m,n)=\sum_{H\le \Z_m \times \Z_n} |H|^t,
	\end{equation*}
	which is multiplicative. If $t=0$, then $\sigma_0(m,n)=s(m,n)$. Here we prove 
	the following related results. 
	
	\begin{theorem} \label{Th_sigma_t} Let $m,n\in \N$ and $t\in \C$. Then
		\begin{align} \label{form_t_1}
			\sigma_t(m,n) & = \sum_{i\mid m, j\mid n} \lcm(i,j)^t (\id_t*\varphi)(\gcd(i,j))
			\\  \label{form_t_2}
			& = \sum_{d \mid \gcd(m,n)} \varphi(d) d^t \sigma_t\left(\frac{m}{d} \right)
			\sigma_t \left(\frac{n}{d} \right).
		\end{align}
	\end{theorem}
	
	\begin{proof} We use the representation of the subgroups of $\Z_m \times \Z_n$, given by \eqref{def_K}. 
		We obtain
		\begin{equation} \label{S_J}
			\sigma_t(m,n) = \sum_{(a,b,c,d,\ell)\in J_{m.n}} |K_{a,b,c,d,\ell}|^t 
			= \sum_{\substack{a\mid m\\ b\mid a}} \sum_{\substack{c\mid
					n\\ d\mid c}} \sum_{a/b=c/d=e} \sum_{\substack{\ell \le e\\ (\ell,e)=1}} (ad)^t
		\end{equation}
		
		Let $m=ax$, $a=by$, $n=cz$, $c=dt$. Rearranging the terms of \eqref{S_J},
		\begin{equation}  \label{proof_line}
			\sigma_t(m,n)= \sum_{bxe=m} \sum_{dze=n} \varphi(e) (bde)^t =
			\sum_{\substack{ix=m\\ jz=n}} i^t \sum_{\substack{be=i\\ de=j}} (j/e)^t \varphi(e)
		\end{equation}
		\begin{equation*}
			= \sum_{\substack{i\mid m\\ j\mid n}} (ij)^t \sum_{e\mid \gcd(i,j)} \frac{\varphi(e)}{e^t} 
			= \sum_{\substack{i\mid m\\ j\mid n}} \frac{(ij)^t}{\gcd(i,j)^t} (\id_t * \varphi)(\gcd(i,j)),
		\end{equation*}
		using that for every $n\in \N$,
		\begin{equation*}
			\sum_{e\mid n} \frac{\varphi(e)}{e^t} =\frac1{n^t} \sum_{e\mid n} e^t\ \varphi(n/e) = \frac1{n^t}
			(\id_t* \varphi)(n),
		\end{equation*}
		and finishing the proof of \eqref{form_t_1}. Now write \eqref{proof_line} as follows:
		\begin{equation*}
			\sigma_t(m,n)= \sum_{\substack{ek=m\\ e\ell=n}} \varphi(e)e^t 
			\sum_{\substack{bx=k\\ dz=\ell}} (bd)^t = \sum_{\substack{ek=m\\
					e\ell=n}} \varphi(e) e^t \sum_{b\mid k} b^t \sum_{d\mid \ell} d^t
		\end{equation*}
		\begin{equation*}
			= \sum_{e\mid \gcd(m,n)} \varphi(e)e^t \sigma_t \left(\frac{m}{e}\right) \sigma_t \left(\frac{n}{e}\right),
		\end{equation*} 
		giving \eqref{form_t_2}.    
	\end{proof}
	
	Note that in the case $t=0$, \eqref{form_t_1} and \eqref{form_t_2} recover identities
	\eqref{total_number_subgroups} and \eqref{total_number_subgroups_var}, respectively.
	
	\begin{theorem} \label{Th_sigma_t_Dirichlet} We have the Dirichlet series representation
		\begin{equation} \label{Dir_sigma_t}
			\sum_{m,n=1}^{\infty} \frac{\sigma_t(m,n)}{m^z n^w}=
			\frac{\zeta(z)\zeta(z-t)\zeta(w)\zeta(w-t)\zeta(z+w-t-1)}{\zeta(z+w-t)}.
		\end{equation}
		and for $t=1$,
		\begin{equation} \label{sigma_1}
			\sum_{m,n\le x} \sigma_1(m,n) = \frac{\pi^6}{864\zeta(3)}x^4+ O\left(x^3 (\log x)^{5/3} \right).
		\end{equation} 
	\end{theorem}
	
	\begin{proof} Here \eqref{Dir_sigma_t} is a direct consequence of identity \eqref{form_t_2}.
		If $t=1$, then using \eqref{form_t_2} we have
		\begin{equation*}
			\sum_{m,n\le x} \sigma_1(m,n) = \sum_{\substack{ek\le x, e\ell\le x}} \varphi(e)e \sigma(k)\sigma(\ell)
			= \sum_{e\le x} \varphi(e)e \left(\sum_{k\le x/e} \sigma(k)\right)^2,
		\end{equation*}
		and \eqref{sigma_1} follows applying formula \eqref{error_sigma}, with Walfisz' error term, and using standard arguments.    
	\end{proof}
	
	\section{Number of all subgroups in the case $k=3$}
	
	Now consider the group $\Gamma :=\Z_m \times \Z_n \times \Z_r$, where $m,n,r\in \N$. 
	By using only simple group-theoretic and number-theoretic arguments,
	Hampejs and T\'oth \cite{HamTot2013} proved the following
	result concerning the representation of the subgroups
	of $\Gamma$. 
	\vskip2mm
	
	{\sl 
		(i) Choose $a,b,c\in \N$ such that $a\mid m, b\mid n, c\mid r$.
		
		(ii) Compute $A:= \gcd(a,n/b)$, $B:=\gcd(b,r/c)$, $C:=\gcd(a,r/c)$.
		
		(iii) Compute
		\begin{equation*}
			X:= \frac{ABC}{\gcd(a(r/c),ABC)}.
		\end{equation*}
		
		(iv) Let $s:=at/A$, where $0\le t\le A-1$.
		
		(v) Let
		\begin{equation*}
			v:=\frac{bX}{B\gcd(t,X)}w, \text{ where } 0\le w\le B\gcd(t,X)/X-1.
		\end{equation*}
		
		(vi) Find a solution $u_0$ of the linear congruence
		\begin{equation*}
			(r/c)u \equiv rvs/(bc) \quad  \text{(mod $a$)}.
		\end{equation*}
		
		(vii) Let $u:=u_0+az/C$, where $0\le z\le C-1$.
		
		(viii) Consider
		\begin{equation*}
			U_{a,b,c,t,w,z}:= \langle (a,0,0),(s,b,0),(u,v,c) \rangle
		\end{equation*}
		\begin{equation*}
			= \{(ia+js+ku, jb+kv,kc): 0\le i\le n/a-1, 0\le j\le n/b-1, 0\le
			k\le n/c-1 \}.
		\end{equation*}
		
		Then $U_{a,b,c,t,w,z}$ is a subgroup of order $mnr/(abc)$ of 
		$\Gamma$. Moreover, there is a bijection between the set of
		sixtuples $(a,b,c,t,w,z)$ satisfying conditions (i)-(viii) and
		the set of subgroups of $\Gamma$.
	}
	\vskip2mm
	
	Let $P(n)=\sum_{k=1}^n \gcd(k,n)=\sum_{d\mid n} d\varphi(n/d)$ be the
	gcd-sum function. Note that the function $P$ is multiplicative and
	$P(p^\nu)= (\nu+1)p^\nu - \nu p^{\nu -1}$ for every prime power 
	$p^\nu$ ($\nu \in \N$). 
	
	For every $m,n,r\in \N$ the total number of subgroups
	of the group $\Z_m \times \Z_n \times \Z_r$ is given by
	\begin{equation} \label{form_3}
		s(m,n,r)= \sum_{a\mid m, b\mid n, c\mid r} \frac{ABC}{X^2} P(X),
	\end{equation}
	where $A,B,C$ are defined above. The number of subgroups of order $\delta$ ($\delta \mid mnr$) is
	given by \eqref{form_3} with the additional condition that the
	summation is subject to $abc= mnr/\delta$.
	If one of $m,n,r$ is $1$, then formula \eqref{form_3} 
	reduces to \eqref{total_number_subgroups}.
	
	It follows that for every prime $p$ and every $\lambda_1,\lambda_2,\lambda_3 \in \N$,
	$s(p^{\lambda_1},p^{\lambda_2},p^{\lambda_3})$ is a polynomial in $p$ with
	integer coefficients. Using different arguments, namely the formulas in terms of the gaussian coefficients
	mentioned in the Introduction, it was shown by T\u{a}rn\u{a}uceanu and T\'oth \cite[Cor.\ 2.2]{TarTot2017}
	that the corresponding explicit formula is, with $\lambda_1 \ge \lambda_2 \ge \lambda_3\ge 1$, 
	\begin{equation*}
		s(p^{\lambda_1},p^{\lambda_2},p^{\lambda_3}) = \frac{F(p)}{(p^2-1)^2(p-1)},
	\end{equation*}
	where
	\begin{align*}
		F(p) = & (\lambda_3+1)(\lambda_1-\lambda_2+1)p^{\lambda_2+\lambda_3+5} +
		2(\lambda_3+1)p^{\lambda_2+\lambda_3+4}  \\
		&  - 2(\lambda_3+1)(\lambda_1-\lambda_2)p^{\lambda_2+\lambda_3+3} -
		2(\lambda_3+1)p^{\lambda_2+\lambda_3+2} \\
		&+ (\lambda_3+1)(\lambda_1-\lambda_2-1)p^{\lambda_2+\lambda_3+1} -
		(\lambda_1+\lambda_2-\lambda_3+3)p^{2\lambda_3+4} \\
		& -2 p^{2\lambda_3+3} + (\lambda_1 + \lambda_2 - \lambda_3-1) p^{2\lambda_3+2}\\
		& + (\lambda_1 +\lambda_2 + \lambda_3+5) p^2 + 2p -(\lambda_1 + \lambda_2 + \lambda_3 +1).
	\end{align*}
	
	This formula was also obtained by Oh \cite[Cor.\ 2.2]{Oh2013}, using different arguments.
	
	The final result of this paragraph is an asymptotic formula for the number $s(n,n,n)$ of subgroups of $\Z_n^3$. Define the multiplicative function $h$ by
	\begin{equation*} 
		s(n,n,n)= \sum_{d\mid n} d^2\tau(d) h(n/d) \quad (n\in \N)
	\end{equation*}
	and let $H(z)= \sum_{n=1}^{\infty} h(n)n^{-z}$ be the Dirichlet
	series of $h$.
	
	We have, see Hampejs and T\'oth \cite[Th.\ 2.3]{HamTot2013} that
	for every $\varepsilon >0$,
	\begin{equation} \label{asympt_s_n_n_n}
		\sum_{n\le x} s(n,n,n) = \frac{x^3}{3} \left(H(3)( \log x+ 2\gamma-1)+
		H'(3) \right) + O\left(x^{2+\theta+ \varepsilon}\right),
	\end{equation}
	where $H'(z)$ is the derivative of $H(z)$, $\gamma$ is Euler's constant and $\theta$ is the exponent 
	in Dirichlet's divisor problem.
	
	\section{Number of all subgroups in the case $k=4$} \label{Section_k_4}
	
	Now consider the group $\Z_m\times \Z_n \times \Z_r \times \Z_s$, and let
	$N(m,n,r,s)$ and $N(m,n,r,s;k)$ denote the total number of its subgroups and 
	the number of its subgroups of order $k$, respectively.
	
	We state and prove the following results, already given in the arXiv paper by the author \cite{Tot2016}.
	Namely, we deduce direct formulas for $N(m,n,r,s)$ and $N(m,n,r,s;k)$, where $m,n,r,s,k\in \N$, $k\mid mnrs$, see Theorems \ref{Theorem_subgroups} and 
	\ref{Theorem_subgroups_order_k}. As direct consequences,
	in Corollaries \ref{Cor_subgroups} and \ref{Cor_subgroups_order_k} we derive formulas for
	the number of subgroups of $p$-groups of rank four. We do not obtain the explicit form of the corresponding polynomials, but those can be computed in 
	special cases. We present some tables for such polynomials, computed using Corollaries \ref{Cor_subgroups} and \ref{Cor_subgroups_order_k}. For the proofs 
	we use the representation of the subgroups of $\Z_m \times \Z_n$, Goursat's lemma (see Section \ref{Section_k_2}) and Proposition \ref{lemma_quotient}.
	A table of values for $N(n):=N(n,n,n,n)$ is also included.
	
	\begin{theorem} \label{Theorem_subgroups} For every $m,n,r,s\in \N$ we have
		\begin{align} \label{sum_4}
			N(m,n,r,s)= \sum \varphi(x_3) \varphi(y_3) \varphi(z_3) \varphi(t_3) F(u,v),
		\end{align}
		where $\varphi$ is Euler's function, $F(u,v)$ represents the number of automorphisms of the group $\Z_u \times \Z_v$, and the sum is over all 
		\begin{equation*}
			(x_1,x_2,x_3,x_4,x_5,y_1,y_2,y_3,y_4,y_5,z_1,z_2,z_3,z_4,z_5,t_1,t_2,t_3,t_4,t_5)\in
			\N^{20}
		\end{equation*} 
		satisfying the following $10$ conditions:
		
		{\rm (1)}\ $x_1x_2x_3=m$,\quad
		
		{\rm (2)}\ $x_3x_4x_5=n$,\quad
		
		{\rm (3)}\ $y_1y_2y_3=\gcd(x_1,x_4)$, \quad
		
		{\rm (4)}\ $y_3y_4y_5=x_3\lcm(x_1,x_4)$,
		
		{\rm (5)}\ $z_1z_2z_3=r$, \quad
		
		{\rm (6)}\ $z_3z_4z_5=s$, \quad
		
		{\rm (7)}\ $t_1t_2t_3=\gcd(z_1,z_4)$, \quad
		
		{\rm (8)}\ $t_3t_4t_5=z_3\lcm(z_1,z_4)$,
		
		{\rm (9)}\ $\gcd(y_2,y_5)=\gcd(t_2,t_5)=:u$, \quad
		
		{\rm (10)}\ $y_3\lcm(y_2,y_5)= t_3\lcm(t_2,t_5)=:v$.
	\end{theorem}
	
	\begin{proof}
		We apply Goursat's lemma for the groups $G=\Z_m \times \Z_n$ and
		$H=\Z_r \times \Z_s$. Let
		\begin{equation*}
			A=K_{a,b,c,d,\ell} \le \Z_m \times \Z_n,
		\end{equation*}
		where $a\mid m, b\mid a, c\mid n, d\mid c,
		\frac{a}{b}=\frac{c}{d}=e, 1\le \ell \le e, \gcd \left(\ell,e
		\right)=1$. Here, by \eqref{H_isom},
		\begin{equation*}
			A \simeq \Z_{m_1} \times \Z_{n_1},
		\end{equation*}
		where $m_1=\gcd(b,d)$, $n_1=\lcm(a,c)$. Also, let
		\begin{equation*}
			B \simeq K_{a_1,b_1,c_1,d_1,\ell_1} \le A,
		\end{equation*} 
		where $a_1\mid m_1, b_1\mid a_1, c_1\mid n_1, d_1\mid c_1,
		\frac{a_1}{b_1}=\frac{c_1}{d_1}=e_1, 1 \le \ell_1 \le e_1, \gcd
		\left(\ell_1,e_1 \right)=1$. We have, by using Proposition
		\ref{lemma_quotient},
		\begin{equation} \label{A/B}
			A/B \simeq \Z_u \times \Z_v,
		\end{equation}
		where $u=\gcd(m_1/a_1,n_1/c_1)$, $v=\lcm(m_1/b_1,n_1/d_1)$.
		
		In a similar way, let
		\begin{equation*}
			C=K_{\alpha,\beta,\gamma,\delta,\lambda} \le \Z_r \times \Z_s,
		\end{equation*}
		where $\alpha \mid r, \beta \mid \alpha, \gamma \mid s, \delta \mid
		\gamma, \frac{\alpha}{\beta}=\frac{\gamma}{\delta}=\epsilon, 1\le
		\lambda \le \epsilon, \gcd \left(\lambda,\epsilon \right)=1$. Here,
		\begin{equation*}
			C \simeq \Z_{r_1} \times \Z_{s_1},
		\end{equation*}
		where $r_1=\gcd(\beta,\delta)$, $s_1=\lcm(\alpha,\gamma)$.
		Furthermore, let
		\begin{equation*}
			D \simeq K_{\alpha_1,\beta_1,\gamma_1,\delta_1,\lambda_1} \le C,
		\end{equation*}
		with $\alpha_1\mid r_1, \beta_1\mid \alpha_1, \gamma_1\mid s_1,
		\delta_1\mid \gamma_1,
		\frac{\alpha_1}{\beta_1}=\frac{\gamma_1}{\delta_1}=\epsilon_1, 1 \le
		\lambda_1 \le \epsilon_1, \gcd \left(\lambda_1,\epsilon_1
		\right)=1$. Here,
		\begin{equation} \label{C/D}
			C/D \simeq \Z_{\eta} \times \Z_{\theta},
		\end{equation}
		where $\eta= \gcd(r_1/\alpha_1,s_1/\gamma_1)$,
		$\theta=\lcm(r_1/\beta_1,s_1/\delta_1)$.
		
		According to \eqref{A/B} and \eqref{C/D}, the quotient groups $A/B$ and $C/D$ are isomorphic.
		if and only if $u=\eta$ and $v=\theta$. In this case there are $F(u,v)$ isomorphisms $A/B \simeq C/D$.
		
		It follows that
		\begin{equation*}
			N(m,n,r,s)=\sum F(u,v),
		\end{equation*}
		where the sum is over all 
		\begin{equation*}
			(a,b,c,d,\ell,a_1,b_1,c_1,d_1,\ell_1, \alpha,\beta,\gamma,\delta,\lambda,\alpha_1,\beta_1,\gamma_1,\delta_1,\lambda_1)\in
			\N^{20}
		\end{equation*}
		satisfying the above properties.
		
		More exactly, let $m=ax, a=by$, $n=cz, c=dt$, where $y=t=e$. Hence $m=bxe$,
		$n=dze$. Also, $m_1=b_1x_1e_1$, $n_1=d_1z_1e_1$, where
		$m_1=\gcd(b,d)$, $n_1=\lcm(a,c)= e\lcm(b,d)$. Similarly, let $r=\alpha X, \alpha=\beta Y$, $s=\gamma Z,
		\gamma=\delta T$, where $Y=T=\epsilon$. Hence $r=\beta X\epsilon$, $s=\delta
		Z\epsilon$. Also, $r_1=\beta_1X_1\epsilon_1$,
		$s_1=\delta_1Z_1\epsilon_1$, where $r_1=\gcd(\beta,\delta)$,
		$r_1=\lcm(\alpha,\gamma)= \epsilon \lcm(\beta,\delta)$.
		
		We deduce that
		\begin{equation*}
			N(m,n,r,s)=\sum F(u,v),
		\end{equation*}
		where the sum is over all 
		\begin{equation*}
			(b,x,e,d,z,\ell, b_1,x_1,e_1,d_1,z_1,\ell_1, \beta,X,\epsilon,\delta,Z,\lambda, \beta_1,X_1,\epsilon_1,
			\delta_1,Z_1,\lambda_1)\in \N^{24}
		\end{equation*}
		satisfying the following conditions:
		
		{\rm (1a)}\ $bxe=m$,\quad
		
		{\rm (2a)}\ $dze=n$,\quad
		
		{\rm (3a)}\ $b_1x_1e_1=\gcd(b,d)$, \quad
		
		{\rm (4a)}\ $d_1z_1e_1= e \lcm(b,d)$,
		
		{\rm (4b)}\ $1\le \ell \le e, \gcd(\ell,e)=1$,
		
		{\rm (5a)}\ $\beta X \epsilon=r$, \quad
		
		{\rm (6a)}\ $\delta Z\epsilon=s$, \quad
		
		{\rm (7a)}\ $\beta_1X_1\epsilon_1=\gcd(\beta,\delta)$, \quad
		
		{\rm (8a)}\ $\delta_1Z_1\epsilon_1=\epsilon\lcm(\beta,\delta)$,
		
		{\rm (8b)}\ $1\le \ell_1 \le e_1, \gcd(\ell_1,e_1)=1$,
		
		{\rm (9a)}\ $\gcd(x_1,z_1)=\gcd(X_1,Z_1)=u$, \quad
		
		{\rm (10a)}\ $e_1\lcm(x_1,z_1)= \epsilon_1\lcm(X_1,Z_1)=v$,
		
		{\rm (10b)}\ $1\le \lambda \le \epsilon, \gcd(\lambda,\epsilon)=1$,
		
		{\rm (10c)}\ $1\le \lambda_1 \le \epsilon_1, \gcd(\lambda_1,\epsilon_1)=1$.
		
		Note that, according to (4b), (8b), (10b) and (10c), the variables $\ell$, $\ell_1$, $\lambda$ and $\lambda_1$ take $\varphi(e)$, $\varphi(e_1)$,
		$\varphi(\epsilon)$, respectively $\varphi(\epsilon_1)$ distinct values, independently from the other variables.
		
		Now formula \eqref{sum_4} follows by introducing the notation
		
		$x_1:=b$, $x_2:=x$, $x_3:=e$, $x_4:=d$, $x_5:=z$,
		$y_1:=b_1$, $y_2:=x_1$, $y_3:=e_1$, $y_4:=d_1$, $y_5:=z_1$
		$z_1:=\beta$, $z_2:=X$, $z_3:=\epsilon$, $z_4:=\delta$, $z_5:=Z$,
		$t_1:=\beta_1$, $t_2:=X_1$, $t_3:=\epsilon_1$, $t_4:=\delta_1$, $t_5:=Z_1$.
	\end{proof} 
	
	We remark that the number $F(m,n)$ of automorphisms of the group $\Z_m \times \Z_n$, used above, 
	can be computed as follows. For every $m,n\in \N$,
	\begin{equation*}
		F(m,n)=\prod_p F(p^{\nu_p(m)}, p^{\nu_p(n)}),
	\end{equation*}
	hence $(m,n)\mapsto F(m,n)$ is a multiplicative arithmetic
	function of two variables. For prime powers $p^a, p^b$ one has, cf. \cite[Th.\
	4.1]{Hil}, \cite[Cor.\ 3]{GG2008},
	\begin{align*}
		F(p^a,p^b)  = \begin{cases} p^{2a} \varphi(p^a) \varphi(p^b), & \text{if $0\le a<b$;} \\ p^a \varphi_2(p^a) \varphi(p^a), & \text{if $0\le a=b$;}
		\end{cases}
	\end{align*}
	where $\varphi_2$ is the Jordan function of order $2$ given by $\varphi_2(n)=n^2\prod_{p\mid n} (1-\frac1{p^2})$. We deduce that
	$F(1,p^b)=\varphi(p^b)$ for $b\in \N_0$ and
	\begin{align*}
		F(p^a,p^b)  =\begin{cases} p^{3a+b} \left(1-\frac1{p}\right)^2, & \text{if $1\le a<b$;} \\
			p^{4a} \left(1-\frac1{p}\right)\left(1-\frac1{p^2}\right), & \text{if $1\le a=b$.}
		\end{cases}
	\end{align*}
	
	\begin{theorem} \label{Theorem_subgroups_order_k} For every $k,m,n,r,s\in
		\N$ such that $k\mid mnrs$, the number $N(m,n,r$, $s;k)$ is given by
		the sum \eqref{sum_4} with the additional constraint
		
		{\rm (11)} \ $x_1x_3x_4t_1t_3t_4=k$.
	\end{theorem}
	
	\begin{proof}
		According to the second part of Goursat's lemma, see Section \ref{Section_k_2}, we have 
		$k=|A||D|=|B||C|$. Here $|A|=ad=bde$, $|B|=a_1d_1=b_1d_1e_1$. Similarly,
		$|C|=\alpha \delta = \beta \delta \epsilon$, $|D|=\alpha_1 \delta_1= \beta_1 \delta_1 \epsilon_1$.
		Hence condition $k=|A||D|=|B||C|$ is equivalent to $bde\beta_1 \delta_1 \epsilon_1=k$ and to $x_1x_3x_4t_1t_3t_4=k$ with the changing of
		notation of above. Note that relation $k=|B||C|$ is a consequence of the other constraints and needs not be considered.
	\end{proof}
	
	\begin{corollary} \label{Cor_subgroups} For every $a,b,c,d\in \N_0$ we have
		\begin{align} \label{sum_4_primes}
			N(p^a,p^b,p^c,p^d)= \sum \varphi(p^{x_3}) \varphi(p^{y_3}) \varphi(p^{z_3}) \varphi(p^{t_3})
			F(p^u,p^v),
		\end{align}
		where the sum is over all
		\begin{equation*}
			(x_1,x_2,x_3,x_4,x_5,y_1,y_2,y_3,y_4,y_5,z_1,z_2,z_3,z_4,z_5,t_1,t_2,t_3,t_4,t_5)\in
			\N_0^{20} 
		\end{equation*}
		satisfying the following $10$ conditions:
		
		{\rm (i)}\ $x_1+x_2+x_3=a$,\quad
		
		{\rm (ii)}\ $x_3+x_4+x_5=b$,\quad
		
		{\rm (iii)}\ $y_1+y_2+y_3=\min(x_1,x_4)$, \quad
		
		{\rm (iv)}\ $y_3+y_4+y_5=x_3+\max(x_1,x_4)$,
		
		{\rm (v)}\ $z_1+z_2+z_3=c$, \quad
		
		{\rm (vi)}\ $z_3+z_4+z_5=d$, \quad
		
		{\rm (vii)}\ $t_1+t_2+t_3=\min(z_1,z_4)$, \quad
		
		{\rm (viii)}\ $t_3+t_4+t_5=z_3+\max(z_1,z_4)$,
		
		{\rm (ix)}\ $\min(y_2,y_5)=\min(t_2,t_5)=:u$, \quad
		
		{\rm (x)}\ $y_3+\max(y_2,y_5)= t_3+ \max(t_2,t_5)=:v$.
	\end{corollary}
	
	\begin{proof}
		Apply Theorem \ref{Theorem_subgroups} in the case $(m,n,r,s)=(p^a,p^b,p^c,p^d)$.
	\end{proof}
	
	\begin{corollary} \label{Cor_subgroups_order_k} For every $a,b,c,d,k\in \N_0$
		with $k\le a+b+c+d$ the number $N(p^a,p^b,p^c,p^d;p^k)$ is given by
		the sum \eqref{sum_4_primes} with the additional constraint
		
		{\rm (xi)} \ $x_1+x_3+x_4+t_1+t_3+t_4=k$.
	\end{corollary}
	
	\begin{proof}
		Apply Theorem \ref{Theorem_subgroups_order_k} in the case $(m,n,r,s)=(p^a,p^b,p^c,p^d)$, $k:=p^k$.
	\end{proof}
	
	Consider the polynomials $N(p^a,p^b,p^c,p^d)$ and $N(p^a,p^b,p^c,p^d;p^k)$, where $p$ is prime, 
	$1\le a\le b \le c \le d$, $0\le k \le a+b+c+d=:n$. It is known that $N(p^a,p^b,p^c,p^d;p^k)= 
	N(p^a,p^b,p^c,p^d;p^{n-k})$ for every
	$0\le k\le n$. The following unimodality result is also known: $N(p^a,p^b,p^c,p^d;p^k) - N(p^a,p^b,p^c,p^d;p^{k-1})$ has non-negative coefficients for every $1\le k \le \lfloor n/2 \rfloor$. More 
	generally, similar properties hold for every abelian $p$-group.
	See Butler \cite{But1987}, Takegahara \cite{Tak1998}.
	
	In what follows we give the polynomials representing the values of $N(p^a,p^b$, $p^c,p^d)$ and $N_{p^k}(p^a,p^b,p^c,p^d)$,
	where $1\le a\le b \le c \le d\le 2$, also $a=b=c=d=3$, with $0\le k\le \lfloor n/2 \rfloor$, computed using Corollaries \ref{Cor_subgroups} and \ref{Cor_subgroups_order_k}.
	Some more tables are given in \cite{Tot2016}.

	\[
	\vbox{\offinterlineskip  \hrule \halign{ \strut \vrule  $\ # \
			$ \hfill & \vrule  $\ #  $ \hfill  \vrule \cr
			N(p,p,p,p) & 5+ 3p+4p^2+3p^3+p^4 \cr \noalign{\hrule \hrule}
			N(p,p,p,p;1)& 1  \cr \noalign{\hrule}
			N(p,p,p,p;p) & 1+p+ p^2+p^3  \cr \noalign{\hrule}
			N(p,p,p,p;p^2) & 1+p+2p^2+p^3+p^4 \cr} \hrule}
	\]
	
	\[
	\vbox{\offinterlineskip  \hrule \halign{ \strut \vrule  $\ # \
			$ \hfill & \vrule  $\ #  $ \hfill  \vrule \cr
			N(p,p,p,p^2) & 6+ 4p+6p^2+6p^3+2p^4 \cr \noalign{\hrule \hrule}
			N(p,p,p,p^2;1)& 1  \cr \noalign{\hrule}
			N(p,p,p,p^2;p) & 1+p+ p^2+p^3  \cr \noalign{\hrule}
			N(p,p,p,p^2;p^2) & 1+p+2p^2+2p^3+p^4 \cr} \hrule}
	\]
	
	\[
	\vbox{\offinterlineskip  \hrule \halign{ \strut \vrule  $\ # \
			$ \hfill & \vrule  $\ #  $ \hfill  \vrule \cr
			N(p,p,p^2,p^2) & 7+ 5p+8p^2+9p^3+6p^4+p^5 \cr \noalign{\hrule \hrule}
			N(p,p,p^2,p^2;1)& 1  \cr \noalign{\hrule}
			N(p,p,p^2,p^2;p) & 1+p+ p^2+p^3  \cr \noalign{\hrule}
			N(p,p,p^2,p^2;p^2) & 1+p+2p^2+2p^3+2p^4 \cr \noalign{\hrule}
			N(p,p,p^2,p^2;p^3) & 1+p+2p^2+3p^3+2p^4+p^5  \cr} \hrule}
	\]
	
	\[
	\vbox{\offinterlineskip  \hrule \halign{ \strut \vrule  $\ # \
			$ \hfill & \vrule  $\ #  $ \hfill  \vrule \cr
			N(p,p^2,p^2,p^2) & 8+ 6p+10p^2+12p^3+10p^4+6p^5+2p^6 \cr \noalign{\hrule \hrule}
			N(p,p^2,p^2,p^2;1)& 1  \cr \noalign{\hrule}
			N(p,p^2,p^2,p^2;p) & 1+p+ p^2+p^3  \cr \noalign{\hrule}
			N(p,p^2,p^2,p^2;p^2) & 1+p+2p^2+2p^3+2p^4+p^5 \cr \noalign{\hrule}
			N(p,p^2,p^2,p^2;p^3) & 1+p+2p^2+3p^3+3p^4+2p^5+p^6 \cr} \hrule}
	\]
	
	\[
	\vbox{\offinterlineskip  \hrule \halign{ \strut \vrule  $\ # \
			$ \hfill & \vrule  $\ #  $ \hfill  \vrule \cr
			N(p^2,p^2,p^2,p^2) & 9+ 7p+12p^2+15p^3+14p^4+11p^5+9p^6+3p^7+p^8 \cr \noalign{\hrule \hrule}
			N(p^2,p^2,p^2,p^2;1)& 1  \cr \noalign{\hrule}
			N(p^2,p^2,p^2,p^2;p) & 1+p+ p^2+p^3  \cr \noalign{\hrule}
			N(p^2,p^2,p^2,p^2;p^2) & 1+p+2p^2+2p^3+2p^4+p^5+p^6 \cr \noalign{\hrule}
			N(p^2,p^2,p^2,p^2;p^3) & 1+p+2p^2+3p^3+3p^4+3p^5+2p^6+p^7 \cr \noalign{\hrule}
			N(p^2,p^2,p^2,p^2;p^4) & 1+p+2p^2+3p^3+4p^4+3p^5+3p^6+p^7+p^8 \cr} \hrule}
	\]
	
	\[
	\vbox{\offinterlineskip  \hrule \halign{ \strut \vrule  $\ # \ $
			\hfill & \vrule  $\ #  $ \hfill  \vrule \cr
			N(p^3,p^3,p^3,p^3) & 13+
			11p+20p^2+27p^3+30p^4+31p^5+37p^6+30p^7 \cr
			& +26p^8 +18p^9 +9p^{10}+3p^{11}+p^{12}
			\cr \noalign{\hrule \hrule}
			N(p^3,p^3,p^3,p^3;1)& 1  \cr \noalign{\hrule}
			N(p^3,p^3,p^3,p^3;p) & 1+p+ p^2+p^3  \cr \noalign{\hrule}
			N(p^3,p^3,p^3,p^3;p^2) & 1+p+2p^2+2p^3+2p^4+p^5+p^6 \cr \noalign{\hrule}
			N(p^3,p^3,p^3,p^3;p^3) & 1+p+2p^2+3p^3+3p^4+3p^5+3p^6+2p^7+p^8+p^9 \cr \noalign{\hrule}
			N(p^3,p^3,p^3,p^3;p^4) & 1+p+2p^2+3p^3+4p^4+4p^5+5p^6+4p^7+4p^8+2p^9 \cr
			& + p^{10} \cr \noalign{\hrule}
			N(p^3,p^3,p^3,p^3;p^5) & 1+p+2p^2+3p^3+4p^4+5p^5+6p^6+6p^7+5p^8+4p^9 \cr
			& +2p^{10} +p^{11} \cr \noalign{\hrule}
			N(p^3,p^3,p^3,p^3;p^6) & 1+p+2p^2+3p^3+4p^4+5p^5+7p^6+6p^7+6p^8+4p^9 \cr
			& +3p^{10} + p^{11}+p^{12}\cr} \hrule}
	\]
	
	In our preprint \cite{Tot2016} we have formulated the following conjectures, confirmed by the above examples.
	For every $1\le a\le b \le c \le d$, the degree of the polynomial $N(p^a,p^b,p^c,p^d)$ is $2a+b+c$.
	For every $1\le m$ the degree of the polynomial $N(p^m,p^m,p^m,p^m)$ is $4m$ and its leading coefficient is $1$. 
	
	These conjectures have been answered in the affirmative by Chew, Chin and Lim \cite{CCL2020}. They used some different arguments, namely by constructing a recurrence relation for the number of subgroups of rank four $p$-groups, they considered the number of subgroups contained in the maximal subgroups.
	They also deduced an explicit formula for $N(p^a,p^b,p^c,p^d)$, having main term $(d-c+1)(b-a+1)p^{2a+b+c}$, where $1\le a\le b\le c\le d$. 
	
	Admasu and Sehgal \cite{AS2021} used a different approach and found an explicit formula for 
	the polynomials $N(p^m,p^m,p^m,p^m)$ with $m\in \N$, also proving the second conjecture of above.
	
	The values of $N(n):=N(n,n,n,n)$ for $1\le n\le 30$ are given by the next table.
	
	\[
	\vbox{\offinterlineskip \hrule \halign{ \strut
			\vrule \hfill $\ # \ $ \hfill
			& \vrule \hfill $\ # \ $
			& \vrule \vrule  \hfill $\ # \ $ \hfill
			& \vrule \hfill $\ # \ $
			& \vrule \vrule \hfill $\ # \ $  \hfill
			& \vrule \hfill $\ # \ $  \vrule \cr
			n & N(n) & n & N(n) & n & N(n) \ \cr \noalign{\hrule}
			1 & 1         & 11 & 19\, 156      & 21  & 774\, 224   \cr \noalign{\hrule}
			2 & 67        & 12 & 420\, 396     & 22  & 1\, 283\, 452   \cr \noalign{\hrule}
			3 & 212       & 13 & 35\, 872      & 23  & 318\, 532  \cr \noalign{\hrule}
			4 & 1\, 983   & 14 & 244\, 684     & 24  & 9\, 187\, 868   \cr \noalign{\hrule}
			5 & 1\, 120   & 15 & 237\, 440     & 25 & 810\, 969  \cr \noalign{\hrule}
			6 & 14\, 204  & 16 & 821\, 335     & 26 & 2\, 403\, 424   \cr \noalign{\hrule}
			7 & 3\, 652   & 17 & 99\, 472      & 27 & 2\, 222\, 704 \cr \noalign{\hrule}
			8 & 43\, 339  & 18 & 1\, 610\, 211 & 28 & 7\, 241\, 916  \cr \noalign{\hrule}
			9 & 24\, 033  & 19 & 152\, 404     & 29 & 783\, 904   \cr \noalign{\hrule}
			10 & 75\, 040 & 20 & 2\, 220\, 960 & 30 & 15\, 908\, 480  \cr} \hrule}
	\]
	
	\section{Final remarks}
	
	Identity \eqref{total_number_cyclic_subgroups} by the author is representing the number $c(n_1,\ldots,n_k)$
	of cyclic subgroups of the group $\Z_{n_1} \times \cdots \times \Z_{n_k}$ with arbitrary $k\ge 2$.
	We note that it has also been cited, resp. applied by Kumar \cite{Kum2020,Kum2022}, Plavnik et. al. \cite{PSY2023}.  
	It is not known any similar compact identity for the total number $s(n_1,\ldots,n_k)$ of subgroups. 
	However, Ramar\'e \cite[Eq.\ (20)]{Ram2017} provided an explicit but involved formula to compute $s(n_1,\ldots,n_k)$ and also the corresponding 
	generalized  function \eqref{def_sigma_t}, by using the primary decomposition of  
	$\Z_{n_1} \times \cdots \times \Z_{n_k}$.
	
	We mention two other related known results. Let $G_{(p)}$ be an arbitrary $p$-group of type $\lambda= (\lambda_1,\ldots,\lambda_r)$, that is  $G_{(p)}\simeq \Z_{p^{\lambda_1}}\times \cdots \times \Z_{p^{\lambda_r}}$, where $\lambda_1\ge \ldots \ge \lambda_r\ge 1$.
	T\u{a}rn\u{a}uceanu \cite[Prop.\ 3.2]{Tar2010} proved that the number of maximal subgroups of $G_{(p)}$ is $1+p+\cdots +p^{k-1}$. Chew, Chin and Lim \cite[Prop.\ 5.1]{CCL2020} showed, by using a general recurrence relation for the number of subgroups of rank four, and extending their method to higher rank, that the number of subgroups of $G_{(p)}$ is $1+\lambda_1+\cdots +\lambda_k$ (mod $p$).    
	
	As remarked in the Introduction, an open problem is to find an asymptotic formula for the sum
	$\sum_{n_1,\ldots,n_k\le } s(n_1,\ldots,n_k)$ with $k\ge 3$. See Section \ref{Section_k_2} for known results on $\sum_{m,n\le x} s(m,n)$. 
	It would be also interesting to study asymptotic properties of the one variable function $s(n,\ldots,n)$, and to find an asymptotic formula for the sum $\sum_{n\le x} s(n,\ldots,n)$ with $k\ge 4$. See the known formulas \eqref{asympt_s_n_n} and \eqref{asympt_s_n_n_n} corresponding to the cases $k=2$ and $k=3$, respectively.

\end{document}